\documentclass[reqno]{amsart} 

\usepackage{amssymb}
\usepackage{amsmath}
\usepackage{epic}
\usepackage{accents}
\usepackage{enumitem}
\usepackage{comment}

\usepackage{mathrsfs}

\usepackage{graphicx}
\newcommand*{\Scale}[2][4]{\scalebox{#1}{$#2$}}%

\newtheorem{theorem}{Theorem}[section]
\newtheorem{lemma}[theorem]{Lemma}

\newtheorem{proposition}[theorem]{Proposition}

\newtheorem{corollary}[theorem]{Corollary}
\newtheorem{fact}[theorem]{Fact}

\usepackage{mathtools}

\usepackage{xfrac}  

\def \Z{\mathbb{Z}}
\def \Q{\mathbb{Q}}
\def \N{\mathbb{N}}

\def \ge{\geqslant}

\title[Grothendieck rings of ordered subgroups of $\Q$]{Grothendieck rings of ordered subgroups of $\Q$}

\author[Bhardwaj]{Neer Bhardwaj$^\dagger$ }
\thanks{$^\dagger$Bhardwaj was partially funded by KU Leuven IF C16/23/010}
\address{Department of Mathematics, KU Leuven, B-3001 Leuven, Belgium}
\email{nbhardwaj@msri.org}

\author[Moonen]{Frodo Moonen}
\address{Department of Mathematics, KU Leuven, B-3001 Leuven, Belgium}
\email{frodo.moonen@student.kuleuven.be}
\subjclass[2020]{Primary 03C07, 03C64, 06F20}
\keywords{Archimedean ordered abelian groups, Grothendieck ring}

\begin{document}

\begin{abstract}
Let $G$ be a proper subgroup of $\Q$ and $S_G$ be the set of primes $p$ for which $G$ is $p$-divisible. 
We show that the model-theoretic Grothendieck ring of the ordered abelian  group 
$(G;+,<)$ is a  quotient of  $(\Z/q\Z)[T]/(T+T^2)$, where $q$ is the largest odd integer that divides $p-1$ for all $p \notin S_G$. 

This implies that the Grothendieck ring of $(G;+,<)$ is trivial in various salient cases, for example when $S_G$ is finite, or when $S_G$ does not contain some prime of the form $2^n+1$, $n\in \N$.
\end{abstract}

\maketitle

\section{Introduction}

\noindent
The model-theoretic notion of a  Grothendieck ring  was introduced in \cite{KS, DL} and captures interesting combinatorial and geometric properties of sets definable in a first-order structure.  These objects often play a crucial role in motivic integration \cite{HK, CH1, SV}, and Grothendieck rings of ordered abelian groups are especially relevant in this context. In this paper, we focus on computing the Grothendieck rings of subgroups of $\Q$ in the ordered group language. The edge cases are known: the Grothendieck ring of $(\Z;+,<)$  is trivial \cite{KS}, and  for  the o-minimal structure $(\Q;+,<)$ this ring is isomorphic to $\Z[T]/(T^2+T)$  \cite{KF}.

\subsection*{Grothendieck ring of a structure} 
For an $\mathcal{L}$-structure $\mathcal{M}$, we denote by $\text{Def}(\mathcal{M})$ the family of definable (with parameters) subsets of $\mathcal{M}$. 
The {\em Grothendieck group} of $\mathcal{M}$ is the abelian group generated by symbols $[X]$, $X\in \text{Def}(\mathcal{M})$, with the following relations:\begin{itemize}
    \item $[X]=[Y]$ when $X$ and $Y$ are definably isomorphic,
    \item $[X\cup Y]=[X]+[Y]$ when $X\cap Y=\emptyset$.
\end{itemize}
Extending this group  by multiplication defined by $[X][Y]=[X\times Y]$ we obtain the {\em Grothendieck ring} of $\mathcal{M}$; denoted by $K_0(\mathcal{M})$. We will refer to $[X]$ as the {\em value} of the definable set $X$ in $K_0(\mathcal{M})$ .

\medskip
\noindent
Throughout the paper, $p$ always denotes a prime integer and $G$ is a proper subgroup of $\Q$. We let $S_G$ denote the set of $p$ such that for every $x\in G$ there is $y\in G$ with $p\cdot y=x$. 
Hence $G$ is $p$-divisible if (and only if) $p\in S_G$, and $S_G$  is always a proper subset of the set of integer primes since $G$ is a proper subgroup of $\Q$.
 The focus of this article is to establish the following statement.

\begin{theorem}\label{main}
The Grothendieck ring of the ordered abelian group $(G;+,<)$ is a quotient of $(\Z/q\Z)[T]/(T+T^2)$, where $q$ is the largest odd integer which divides $p-1$ for all $p\notin S_G$.
\end{theorem}

\noindent
We collect various ingredients and complete the proof of this theorem at the end of the paper, although note that  the triviality of the Grothendieck ring of $(G;+,<)$ for various cases is already  established  in Cor~\ref{maincor}. Apart from determining the Grothendieck ring, it is also an important theme to study definable sets up to definable bijection, sometimes even up to some special definable bijections, depending on the context. This leads to some finer classification results, see for example \cite{C1, Cluckers, CH2}, but we  do not address this finer challenge in this paper.


There has been much work on Grothendieck rings of various structures in the language of rings. It follows easily that the   Grothendieck ring of real closed fields is given by $\Z$ \cite{KS}, and in contrast the same paper shows that the Grothendieck ring of the complex field embeds a polynomial ring over $\Z$ generated by continuum many indeterminates.  Cluckers and Haskell showed in \cite{Cluckers_Qp} that the Grothendieck ring of the field of $p$-adic numbers is trivial. Moreover,   the Grothendieck ring of formal Laurent series over $p$-adic numbers and formal Laurent series over local fields of strictly positive characteristic are also trivial by work of Cluckers \cite{Cluckers_Laurent}. 
See also \cite{EE} for another exact computation of the Grothendieck ring of a first-order structure.



\subsection*{Acknowledgements} Our foremost  thanks goes to Pierre Touchard for suggesting this research direction for the second author's Masters dissertation topic. We also express our gratitude to Raf Cluckers for his support and guidance throughout this project. In addition, we are grateful to  Pierre Touchard, Mathias Stout, and Floris Vermeulen for helpful discussions; indeed we attribute the idea for Lemma~\ref{mainL} to these discussions.

\section{Grothendieck ring values of unary definable sets}\label{intervalvalues}

\noindent


Let $G$ be a proper subgroup of $\Q$. As defined in the introduction, $S_G$ is the set of primes $p$ for which $G$ is $p$-divisible. For every positive integer $n$,  we let $n_{S_G}$ denote the positive integer that divides $n$, such that only primes in $S_G$ divide $n_{S_G}$, and no prime in $S_G$ divides $n/n_{S_G}$. Also for $n/n_{S_G}=p_{i_1}^{e_1}\cdots p_{i_m}^{e_m}$, we set $n_{K_G}:=p_{i_1}^{k_{i_1}}\cdots p_{i_m}^{k_{i_m}}$, the product of the prime factors of $n/n_{S_G}$ raised to the power up to which they are invertible in $G$. This means that if  $a$ is the smallest positive integer in $G$, then $a/n_{K_G}\in G$, and if $p$ divides $n_{K_G}$, then $a/(p\cdot n_K)\notin G$.


For any $a,b\in \Z$ with $\text{gcd}(a,b)=1$, it is easy to see that the Grothendieck ring of $G$ and $\frac{a}{b}G$ are the same. Thus it suffices to only consider subgroups of $\Q$ which contain $\Z$. Furthermore, we only need to consider the case where $G$ is dense in $\Q$ with respect to the order topology. Indeed, otherwise there is $a/b\in \Q$ with $\text{gcd}(a,b)=1$ such that $\frac{a}{b}G=\Z$ and the Grothendieck ring of $(G;+,<)$ is trivial. Hence we assume going forward that 
$$ \Z\ \subset G \subset \Q\qquad \text{ and }\qquad G \text{ is dense in } \Q.$$

\subsection{Values of intervals } In the rest of this section,  we  compute certain values realized  by sets of the form $(a,b)\cap G$, where $a,b\in \Q$ with $a<b$. We call such sets {\em $\Q$-intervals} in $G$ and when the context is clear, we abuse notation and use $(a,b)$ to denote the $\Q$-interval  $(a,b)\cap G$. 

These sets are clearly definable in $(G;+,<)$, and  we show in Proposition~\ref{quasi2}  that the Grothendieck ring of this structure is in fact generated by the values of $\Q$-intervals. Hence one major focus of this paper is  to study these values, and we adapt the results of   \cite[Claim 7]{KF} and employ  various number theoretic tricks for this purpose.

\begin{lemma}\label{vals}
The following hold for bounded $\Q$-intervals in $G$.
    \begin{enumerate}[label=(\roman*)]
        \item $[(a,b)]=-1$ if $a,b\in G$.
        \item $[(a,b)]=-\frac{1}{2}$ for some $a\in G$ and $b\in \mathbb{Q}\backslash G$.
    \end{enumerate}
\end{lemma}
\begin{proof}
We first prove $\mathrm{(i)}$. 
For $a, b\in G$ with $a<b$ we have that
$$[(a,\infty)]\ =\ [(a,b)]+ [\{b\}] + [(b, \infty)].$$
Since $b-a\in G$, the map $x\to x+b-a$ gives a definable isomorphism from $(a,\infty)$ to $(b,\infty)$. It follows that $[(a,b)]= -1$ as desired, since the value of a singleton is 1.

For $\mathrm{(ii)}$, recall that $S_G$ does not contain all primes and take $p$ to be the smallest such prime. Suppose $k\in \N$ such that $1/p^k\in G$ and $1/p^{k+1}\notin G$. We get that
\begin{equation*} 
        \begin{split}
        [(0,\frac{1}{p^{k}})]&=[(0,\frac{1}{p^{k+1}})]+[(\frac{1}{p^{k+1}},\frac{1}{p^{k}})] \\
        &= [(0,\frac{1}{p^{k+1}})]+[(0,\frac{p-1}{p^{k+1}})]\\
        &= 2\cdot[(0,\frac{1}{p^{k+1}})]
        \end{split}
\end{equation*}
Here we use that $x\to \frac{1}{p^{k}}-x$ and $x\to x/(p-1)$ are definable isomorphisms. Note by the choice of $p$, all prime factors of $p-1$ are in $S_G$. By  $\mathrm{(i)}$, we know that $[(0,\frac{1}{p^{k}})]=-1$ and hence $[(0,\frac{1}{p^{k+1}})]=-\frac{1}{2}$.
\end{proof}

\begin{lemma} \label{valueslemma}
Set $X:= [(0,+\infty)]$. The following hold for unbounded $\Q$-intervals in $G$.
    \begin{enumerate}[label=(\roman*)]
        \item $[(a,+\infty)] =[(-\infty,a)]=X$ for any $a\in G$.
        \item $[(-\infty,+\infty)] = 2X+1$.
        \item $[(c,+\infty)] =[(-\infty,c)] =X+\frac{1}{2}$ for some $c\in \mathbb{Q}\backslash G$.

        \end{enumerate}
\end{lemma}
\begin{proof}
The assertions in $\mathrm{(i), (ii)}$ are immediate, and  then $\mathrm{(iii)}$  follows  by a direct application of Lemma~\ref{vals}$\mathrm{(ii)}$.
\end{proof}


\begin{lemma}\label{lemmaX^2}
In the Grothendieck ring of $(G;+,<)$ we have that $X^2+X=0$.
\end{lemma}
\begin{proof}
For $I:=(0,\infty)\cap G$, we set $$(f,g)_I:=\{(x,y)\in I\times G:\ f(x)<y<g(x)\},$$ where $f$ and $g$ are definable functions  in $(G;+,<)$, or  $f,g\in \{-\infty, +\infty\}$, with $f(x)<g(x)$ for all $x\in I$. 
We split $I^2$ into three disjoint parts and obtain
\begin{equation}\label{I^2}
    [I^2]=[(0,\text{Id}_{I})_{I}]+[\Gamma(\text{Id}_{I})]+[(\text{Id}_{I},\infty)_I],
\end{equation}
where $\text{Id}_{I}$ is the identity map on $I$ and $\Gamma(\text{Id}_{I})$ denotes its graph. We work with three definable isomorphisms given as follows.
$$f: (0,\text{Id}_{I})_{I}\rightarrow (\text{Id}_{I},\infty)_I: (x,y)\rightarrow (y,x)$$
$$g: I^2\rightarrow (\text{Id}_{I},\infty)_I: (x,y)\rightarrow (x,x+y)$$
$$h: \Gamma(\text{Id}_{I})\rightarrow I: (x,x)\rightarrow x$$
Hence we may  rewrite (\ref{I^2}) as
$$[I^2]=[I^2]+[I]+[I^2],$$
and it follows that $X^2+X=0$ as claimed.
\end{proof}

\subsection{Torsion from primes not in $S_G$}
We show that the primes not in $S_G$ induce a certain precise torsion in the Grothendieck ring of $(G;+,<)$. We shall need the following simple fact.

\begin{fact}\label{div}
Suppose $n\ge 2$ such that $n_{S_G}=1$. Then for all $y\in G$, $n$ divides exactly one of $y, y+\frac{1}{n_{K_G}}, \ldots, y+\frac{n-1}{n_{K_G}}$ in $G$. 
\end{fact}
\begin{proof}
Let $n\ge 2$ such that $n_{S_G}=1$ and   $y=\frac{a}{b}\in G$, with $a$ and $b$ coprime, be given. 
By the definition of $n_{K_G}$, there are $u, v \in \Z$ such that $b=u\cdot v$,  $u$ divides $n_{K_G}$, and $\text{gcd}(v,n_{K_G})=1$. 
The fact that $\text{gcd}(v,n_{K_G})=1$ implies that $\text{gcd}(v,n)=1$ and we have that
$$\frac{a}{b}+\frac{i}{n_{K_G}}\ =\ \frac{a\cdot n_{K_G}/u+vi}{n_{K_G}v}.$$
Since $\text{gcd}(v,n)=1$, there is  unique $i\in \{0,\ldots n-1\}$ such that $$a\cdot n_{K_G}/u+vi\ =\ 0\mod n,$$
and the proof is complete.
\end{proof}


\begin{lemma}\label{mainL}
Consider a prime number $p$ such that  $p\notin S_G$ and let $k\ge 0$ be such that $1/p^k\in G$ and $1/p^{k+1}\notin G$. Then in the Grothendieck ring of $(G;+,<)$ we have that $q=0$ where $q$ is the largest odd factor of $p-1$.
\end{lemma}
\begin{proof}
The idea is to partition $(-\infty,+\infty)\cap G$ into equivalence classes modulo $p$. For $i=0, \ldots, p-1$, set  $N_i\coloneq\{x\in G:\ p\mid (x+\frac{i}{p^k})\}$.  Then by Lemma~\ref{div} we have 
$$2X+1\ =\ [(-\infty,+\infty)]\ =\ \sum_{i=0}^{p-1}[N_i].$$
For $i=0,\ldots,p-1$ we have a definable bijection $N_i \to  G$ given by $x\to\frac{x+\frac{i}{p^k}}{p}$. This implies that
$$2X+1\ =\ p\cdot (2X+1)$$
and hence $(p-1)(2X+1)=0$. Multiplying both sides by $2X+1$ and employing Lemma~\ref{lemmaX^2} we get that $p-1=0$. Since $2$ is invertible in the Grothendieck ring of $(G;+,<)$ by Lemma~\ref{vals}$\mathrm{(ii)}$, the desired conclusion follows.
\end{proof}

\begin{corollary}\label{maincor}

The Grothendieck ring of the ordered abelian group $(G;+,<)$ is trivial if no odd prime  divides $p-1$ for every prime $p\notin S_G$.    
\end{corollary}

\begin{corollary}\label{cases}
The Grothendieck ring of $(G;+,<)$ is trivial in the following cases:
\begin{itemize}
 \item $S_G$ contains only finitely many primes.
    \item $S_G$ does not contain some prime of form $2^n+1$, $n\in \N$. In other words, $S_G$ does not contain $2$ or a Fermat prime.
\end{itemize}
\begin{proof}
The second assertion is immediate. For the first assertion, note that  for any odd integer $q>1$, the arithmetic progression $qm+2$, $m\in \N$, contains infinitely many primes $p$ by the Dirichlet prime number theorem and for any such prime $p$, we have that such that $q\nmid (p-1)$.    
\end{proof}

\end{corollary}

\subsection{Values of intervals revisited} Our next result shows  that the ring generated by the Grothendieck ring values of unary definable sets in $(G;+,<)$ is a quotient of $(\Z/q\Z)[T]/(T+T^2)$, where $q$ is the largest odd integer which divides $p-1$ for all $p\notin S_G$.

\begin{lemma}\label{vals2}
Suppose there is an odd prime $q$ which divides $p-1$ for all primes $p\notin S_G$. Then the following holds for $\Q$-intervals in $G$.
    \begin{enumerate}[label=(\roman*)]
        \item $[(a,b)]=-\frac{1}{2}$ if $a\in G$ and $b\in \mathbb{Q}\backslash G$ or vice versa.
        \item $[(a,b)]=0$ if $a,b\in \mathbb{Q}\backslash G$.
        \item $[(a,+\infty)]=[(-\infty,a)]=X+\frac{1}{2}$ if $a\in \mathbb{Q}\backslash G$.
    \end{enumerate}
\end{lemma}
\begin{proof}
We begin with the proof of $\mathrm{(i)}$. It suffices to consider the case of $a\in G$ and $b\in \Q\setminus G$. 

By translating suitably we may assume that $a=0$ and $b=m/n\notin G$ such that $b<1$ and $m,n\in \N$ with $\text{gcd}(m,n)=1$. Since multiplying and dividing by elements of $S_G$  are definable isomorphisms in $(G;+,<)$, we may assume without loss of generality that all prime factors of $m$ and $n$ are not in $S_G$.
For $b<1/2$, we have that
\begin{equation*}
        \begin{split}
        [(0,1)]&=[(0,\frac{m}{n})]+[(\frac{m}{n},\frac{n-m}{n})]+[(\frac{n-m}{n},1)]\\
        &=2[(0,\frac{m}{n})]+[(\frac{m}{n},\frac{n-m}{n})].\\
        \end{split}
\end{equation*}
We will show that $[(\frac{m}{n},\frac{n-m}{n})]=0$. It then follows that $[(0,\frac{m}{n})]=-1/2$ as desired for $\frac{m}{n}<1/2$. From this, we also obtain the case $b>1/2$ as follows
\begin{equation*}
        \begin{split}
        [(0,\frac{m}{n})]&=[(0,\frac{n-m}{n})]+[(\frac{n-m}{n},\frac{m}{n})]=[(0,\frac{n-m}{n})]=-1/2.
        \end{split}
\end{equation*}

Note that our assumptions imply that $2\in S_G$ and hence $\frac{m}{n}\neq \frac{n-m}{n}$. Suppose $\frac{m}{n}<1/2$. Since $G$ is dense in $\Q$, we have $c\in G$ such that $\frac{m}{n}<c<\frac{n-m}{n}$. Using definable isomorphisms we get that
\begin{equation}\label{split}
        \begin{split}
        [(\frac{m}{n},\frac{n-m}{n})]&=[(\frac{m}{n},c)]+[\{c\}]+ [(c,\frac{n-m}{n})]\\
        &=[(\frac{m}{n}-c,0)]+[(0,\frac{n-m}{n}-c)]+1
        \end{split}
\end{equation}
Next we aim to find $Q\in S_G$ such that $n$ divides $Q+1$ in the integers. Since we can then translate the interval in the first summand of equation~(\ref{split}) 
 by the constant $T=c+Q(1-c) -\frac{(Q+1)m}{n}\in G$ and multiply the second summand of equation~(\ref{split}) by $Q$ respectively to get our desired claim as follows.
\begin{equation*}
        \begin{split}
        [(\frac{m}{n},\frac{n-m}{n})]&= [(Q(1-c)-Q\frac{m}{n},T)]+[(0,Q\big(\frac{n-m}{n}-c\big)]+1\\
        &=[Q\big(\frac{n-m}{n}-c\big),T)]+[(0,Q\big(\frac{n-m}{n}-c\big)]+1\\
        &=[(0,T)]+1\\
        &=0
        \end{split}
\end{equation*}
Here we used Lemma~\ref{vals}($\mathrm{i}$). By our assumption, we have an odd prime $q$ which divides $p-1$ for all $p\notin S_G$, and we solve
 for prime $Q$ satisfying
$$Q\equiv -1 \mod n \qquad \text{and}\qquad  Q\equiv 2 \mod q.$$
Note that $q\in S_G$ and hence we have that $q$  is coprime to $n$ as well. Then by the Chinese remainder theorem and Dirichlet's theorem on primes in arithmetic progressions we have that there is a prime $Q$ satisfying the displayed divisibility conditions. Note such a $Q$ necessarily belongs in $S_G$ since $q$ does not divide $Q-1$. This finishes the proof for $(\mathrm{i})$.

For  $(\mathrm{ii})$, we start with $c,d\in\Q\backslash G$, and take $a,b\in G$ with $a<c<d<b$.  Then  we have that
\begin{equation*}
    -1=[(a,b)]=[(a,c)]+[(c,d)]+[(d,b)]=-\frac{1}{2}+[(c,d)]-\frac{1}{2}.
\end{equation*}
It follows that $[(c,d)]=0$.

Assertion $(\mathrm{iii})$  follows by a direct application of $(\mathrm{i})$ and Lemma  \ref{valueslemma}($\mathrm{i}$).
\end{proof}

\section{Grothendieck ring of $(G;+,<)$ is generated by values of unary definable sets} \label{reduction}

\noindent

It is well-known that the ordered group of integers  admits quantifier elimination in the Presburger language -- $\big(0,1,+,<, (P_n)_{n\in \N^{\ge 1}}\big)$,  where for every $n\in \N^{\ge 1}$, $P_n$ is unary set given by $a\in P_n$ if and only if $\exists x (a=nx)$. In fact any ordered abelian group that does not have a convex definable subgroup also admits quantifier elimination in this language \cite{CH}; in this general setting 1 is defined to be the minimal positive element if such an element exists, and 1 = 0 otherwise.

So we have that the structure $\big(G;0, +, <, (P_n)_{n\in \N^{\ge 1}}\big)$ has quantifier elimination, and we use this fact to  transfer certain aspects of cell decomposition from $(\Q;+,<)$ to $(G;+,<)$. See \cite[Chapter 3]{vdD2} for details about o-minimal cells and o-minimal cell decomposition.

Consider a definable subset $A\subset G^n$. Recall that Fact~\ref{div} gives that $\neg P_{n}(x)$ is equivalent to $\bigvee_{i=1}^{n-1}P_{n}(x+\frac{i}{n_{K_G}})$ for all $n\in \N^{\ge 2}$ with $n_{S_G}=1$. Divisibility conditions for primes in $S_G$ can be dismissed since every element in $G$ is infinitely many times divisible by these primes, therefore the cases we discussed cover all possible kinds of divisibility conditions.

Using this together with quantifier elimination and usual arguments with the Presburger language
we obtain $L\in \N^{\ge 1}$ with $L_{S_G}=1$ such that $A$ is a disjoint union of sets defined by  formulas $\phi_{(\ell_1,\ldots, \ell_n)}^A$ of the following form.
$$\bigvee\ \bigg( \Big(\bigwedge_i \big(\sum_j m_{i_j}x_j=g_i\big)\Big)\ \bigwedge\ \Big(\bigwedge_j P_{L}(x_j-l_j)\Big)\ \bigwedge\ \Big(\bigwedge_i \big(c_i< \sum_j e_{i_j} x_j< d_i\big)\Big) \bigg),$$
where  $\ell_1,\ldots, \ell_n \in \{0,\ldots, \frac {L-1}{L_{K_G}}\}$ and $m_{i_j},e_{i_j}\in \mathbb{Z}$ and $g_i,c_i,d_i\in G$.

Clearly the map $x_j\to (Lz_j+\ell_j)$ gives a definable isomorphism from the  set defined by $\phi_{(\ell_1,\ldots, \ell_n)}^A$, to  the set defined by the following formula
$$\bigvee\ \bigg( \Big(\bigwedge_i \big(\sum_j m_{i_j}(Lz_j+\ell_j)=g_i\big) \Big)\ \bigwedge\ \Big(\bigwedge_i \big(c_i< \sum_j e_{i_j} (Lz_j+i_j)< d_i\big) \Big) \bigg).$$
This formula has no divisibility conditions and hence the set it defines in $\Q^n$ can be partitioned into cells. This implies that in the Grothendieck ring of $(G;+,<)$ we have that
$$[A]=\sum_{(\ell_1,\ldots, \ell_n)} A_{(\ell_1,\ldots, \ell_n)}= \sum_i [C_i\cap G^n],$$
where $C_i$ are cells in $(\Q; +,<)$.

\subsection*{Quasi-cells} For $C$ a cell in $\mathbb{Q}^n$, we say that $C \cap G^{n}$ is a  {\em quasi-cell}. Clearly the quasi-cells are definable in $G$. Our argument above shows that the Grothendieck ring of $(G;+,<)$ is generated by the values of  quasi-cells. As advertised earlier,  something much stronger is true.


\begin{proposition}\label{quasi2}
The Grothendieck ring of $(G;+,<)$ is generated by the values of unary quasi-cells.
\end{proposition}
\begin{proof}
It suffices to show that the value of all quasi-cells can be expressed as a linear combination of values of unary quasi-cells. Given a quasi-cell $C\cap G^n$, we proceed by  induction on $n$. The base-cases $1$ is immediate.

Take a quasi-cell $C'=C\cap G^{n+1}$, where $C$ is a cell in $\Q^{n+1}$. We first consider the case where $C$ is given as $\Gamma(f)$ where $f$ is a definable continuous function on a cell $A\subseteq \mathbb{Q}^{n}$. Let $A'=A\cap G^n$ and we have that
$$C'\ =\ \{(x,y)\in A'\times G:\ \sum_{i=1}^n a_ix_i +b=my\},$$
for some $a_i,b, m\in\Z$. The function $y=\frac{\sum_{i=1}^n a_ix_i+b}{m}$ is only definable on the elements of $A$ divisible by $m$. Let $m_T:=m/m_{S_G}$ and observe that an element of $G$ is divisible by $m$ if and only if it is divisible by $m_T$. The projection map onto the first $n$ co-ordinates $\pi$ gives a definable isomorphism from $C$ to $\pi(C)=A'\cap P_{m_T}(\sum_{i=1}^n a_ix_i+b)$ and we are done by induction.


The other case is that  
$$C\ =\ (f,g)_{A}\ :=\ \{ (x,t):\ x \in A \text{ and } f(x)< t< g(x) \}, $$ where $A$ is a cell in $\mathbb{Q}^{n}$, and 
$$f,g\in \{\text{continuous functions definable in } (\Q;+,<) \text{ with domain } A\} \cup \{ -\infty, \infty\},$$
with   $f(x)<g(x)$ for all  $x\in A$. Set $A'=A\cap G^n$ as before. 

If $f=-\infty$ and $g=+\infty$, then  $C'=A'\times G$ and $[C']=[A']\cdot[G]$. We now proceed to the non-trivial case of $f\neq -\infty$ and $g= +\infty$. We have that
$$C'=\{(x,y)\in A'\times G: \sum_{i=1}^n a_ix_i+b<my\},$$
for some $a_i,b, m\in\Z$.
Let $m_T:=m/m_{S_G}$ as before and observe that 
\begin{equation}\label{new_cell}
    [C']\ =\ \sum_{j=0}^{m_T-1}[\{(x,y)\in A'\times G:\ \sum_{i=1}^n a_ix_i+b<my\}\cap P_{m_T}(\sum_{i=1}^n a_ix_i+b+\frac{j}{m_K})]
\end{equation}
For  each  $j\in \{0,\ldots, m_T-1\}$,  we set $A'_j:=A'\cap P_{m_T}(\sum_{i=1}^n a_ix_i+b+\frac{j}{m_K})$ and $C'_j:=\{(x,y)\in A'_j\times G:\ \sum_{i=1}^n a_ix_i+b<my\}$, and observe that 
$$A'_j\times (\frac{-j}{m\cdot m_K},+\infty)\rightarrow C'_j\ :\ (x,t)\rightarrow (x,\frac{a_ix_i+b + \frac{j}{m_K}}{m}+t)$$
is a definable bijection in $(G;+,<)$. Hence we have that 
$$[C']\ =\ \sum_{j=0}^{m_T-1} [C'_j]\ =\ \sum_{j=0}^{m_T-1} [A'_j]\cdot  [(\frac{-j}{m\cdot m_K},+\infty)],$$ and we are done by induction. The case of  $f=-\infty$ and $g\neq \infty$ is similar.

Finally, suppose $f\neq -\infty$ and $g\neq \infty$.  There is a cell $A\subseteq \Q^n$ such that
$$[C']=[(-\infty,g)_A\cap G^{n+1}]-[\Gamma(f)\cap G^{n+1}]-[(-\infty,f)_A\cap G^{n+1}],$$
and we employ the the previous cases and induction to conclude the proof.
\end{proof}


\noindent
{\em Proof of Theorem~\ref{main}.}
If there is no odd prime which divides $p-1$ for all $p\notin S_G$, then the Grothendieck ring of $(G;+,<)$ is trivial by Theorem~\ref{maincor}. In the remaining case, by proposition ~\ref{quasi2}, we only need to consider the values realized by $\Q$-intervals. Then the desired result follows immediately by Lemmas~\ref{vals}, ~\ref{valueslemma},  ~\ref{lemmaX^2}, ~\ref{mainL}, and~\ref{vals2}.

\bibliographystyle{amsplain}
\bibliography{BM}

\end{document}